\documentclass[11pt,leqno]{amsart}

\usepackage{latexsym,amsmath,amsfonts,amscd}

\newtheorem{theorem}{Theorem}[section]
\newtheorem{lemma}[theorem]{Lemma}
\newtheorem{proposition}[theorem]{Proposition}

\newtheorem{remark}[theorem]{Remark}

\def\r{\mathbb{R}}
\def\rn{\mathbb{R}^N}

\def\rh{\rightharpoonup}
\def\irn{\int_{\rn}}

\def\la{\lambda}
\def\la1{\lambda_1}
\def\d12{\mathcal{D}^{1,2}}
\def\cn{\mathcal{N}}

\numberwithin{equation}{section}

\title[Fractional Schr\"{o}dinger equations ]{Ground state solutions of fractional Schr\"{o}dinger equations with potentials and weak monotonicity condition on the nonlinear term}

\author{Chao Ji}\thanks{Chao Ji was supported by NSFC (grant No. 11301181), China Postdoctoral Science Foundation funded project.}
\address{Center for Applied Mathematics, Tianjin University, 300072 Tianjin, China}

\address{Department of Mathematics, East China University of Science and Technology,  200237 Shanghai, China}
\email{jichao@ecust.edu.cn}



\begin{document}

\baselineskip15pt

\maketitle

\begin{abstract}
In this paper we are concerned with the fractional Schr\"{o}dinger equation $(-\Delta)^{\alpha} u+V(x)u =f(x, u)$, $x\in \rn$, where $f$ is superlinear, subcritical growth and $u\mapsto\frac{f(x, u)}{\vert u\vert}$ is nondecreasing. When $V$ and $f$ are periodic in $x_{1},\ldots, x_{N}$, we show the existence of ground states and the infinitely many solutions if $f$ is odd in $u$. When $V$ is coercive or $V$ has a bounded potential well and $f(x, u)=f(u)$, the ground states are obtained.
When $V$ and $f$ are asymptotically periodic in $x$,  we also obtain the ground states solutions. In the previous research, $u\mapsto\frac{f(x, u)}{\vert u\vert}$ was assumed to be strictly increasing, due to this small change, we are forced to go beyond methods of smooth analysis.
\end{abstract}

\section{Introduction} \label{intro}

In this paper we consider the following fractional Schr\"{o}dinger equation
\begin{equation}  \label{1}
(-\Delta)^{\alpha} u+V(x)u =f(x, u), \quad  x\in \rn,
\end{equation}
where $\alpha\in (0, 1)$, $N> 2\alpha$,  $(-\Delta)^{\alpha}$ stands for the fractional Laplacian, $f\in C(\rn\times \r, \r)$
and the potential $V\in C(\rn, \r)$.\\
When $\alpha=1$, $(1.1)$ becomes the classical Schr\"{o}dinger equation
\begin{equation}  \label{1}
-\Delta u+V(x)u =f(x, u), \quad  x\in \rn.
\end{equation}
 There has been a great deal of works dealing with the equation (1.2). In particular, Szulkin and Weth \cite{rSW} studied the ground state solutions and the infinitely many solutions if  $f(x, u)$ is odd in $u$ for the strong
indefinite case. In their paper, the nonlinear term $f$ satisfies the following assumption:\\
$(F'_{4})$ $u\mapsto\frac{f(x, u)}{\vert u\vert}$ is strictly increasing on $(-\infty, 0)$ and on $(0, \infty)$.\\
They sought the ground states on the generalized Nehari manifold \cite{rP}
\begin{equation*}
\mathcal{M}:=\{u\in E\backslash E^{-}: \Phi'(u)u=0\,\, \text{and}\,\, \Phi'(u)v=0\,\,\text{for all}\,\, v\in E^{-}\},
\end{equation*}
where $H^{1}(\rn):=E=E^{+}\oplus E^{-}$ corresponds to the spectral decomposition of $-\bigtriangleup+V$ with respect to the positive
and negative part of the spectrum and
\begin{equation*}
\Phi(u)=\frac{1}{2}\irn (\vert \nabla u\vert^{2}+V(x)u^{2})dx-\irn F(x, u)dx,
\end{equation*}
$F(x, u)=\int_{0}^{u}f(x, s)ds$.  Because of the assumption $(F'_{4})$,
for any $u\in E\backslash E^{-}$, the set $\mathcal{M}$ intersects $\hat{E}(u):=E^{-}\oplus \mathbb{R}^{+}u=E^{-}\oplus \mathbb{R}^{+}u^{+}$ in exactly
one point $\hat{m}(u)$ which is the unique global maximum point of $\Phi\mid_{\hat{E}(u)}$, the uniqueness of $\hat{m}(u)$ enables one to define a continuous map $u\mapsto \hat{m}(u)$, which is important in the remaining proof.
 If $(F'_{4})$ is replaced by the weaker condition as follows\\
$(F_{4})$ $u\mapsto\frac{f(x, u)}{\vert u\vert}$ is nondecreasing on $(-\infty, 0)$ and on $(0, \infty)$,\\
then $\mathcal{M}\cap\hat{E}(u)$ may be a finite line segment that an example can be seen in \cite{rZZ}, so the argument in \cite{rSW}  collapses. To solve this problem, in \cite{rLi},  by applying linking methods and showing the boundedness of all Cerami sequences for $\Phi$,  Liu obtained the ground states. After that,  via a non-smooth method,  Pavia, Kryszewski and  Szulkin in \cite{rPKS}  gave the ground state solutions and the infinitely many solutions if $f(x, u)$ is odd in $u$, and the result in \cite{rLi}  is an easy consequence of the approach in \cite{rPKS}.  Motivated by \cite{rPKS}, in this paper we will generalize their results to the fractional Schr\"{o}dinger equations when $V$ and $f$ are 1-periodic in $x_{1}$, $\ldots$, $x_{N}$. Since our problem is nonlocal, it is the more difficult and complicated. Moreover, for the coercive potential case, the bounded potential well case, the $V$ and $f$ are asymptotically periodic in $x$ case, we also give the existence of ground states of problem (1.2) via the variational methods \cite{rW}.\\

In recent years, the study of the various nonlinear equations or systems involving fractional Laplacian has received considerable attention. These problems mainly arise in fractional quantum mechanics \cite{rL1, rL2}, physics and chemistry \cite{rMK}, obstacle problems \cite{rSi}, optimization and finance \cite{rCT} and so on. In the remarkable work of Caffarelli and Silvestre \cite{rCS}, the authors express this nonlocal operator $(-\Delta)^{s}$ as a Dirichlet-Neumann map for a certain elliptic boundary value problem with local differential operators defined on the upper half space. This technique is a valid tool to deal with the equations involving fractional operators in the respects of regularity and variational methods. For more results on the fractional differential equations, we refer to \cite{rAP, rCW, rMR, rMRS, rS1, rS2}. Recently, in \cite{rZXZ}, under the assumption $(F'_{4})$ and using Andrzej and Weth's method \cite{rSW}, the authors show the existence of infinitely many solutions of problem (1.2) when $V$ and $f$ are periodic in $x_{1},\ldots, x_{N}$, $f(x, u)$ is odd in $u$. Moveover, when $V$ and $f$ are asymptotically periodic in $x$, they give the ground state solutions. If $(F'_{4})$ is replaced by $(F_{4})$, the argument in \cite{rZXZ} does not work, we will deal with this problem and improve their results.\\

From now on,  we always assume that $\inf_{x\in \rn}V(x)>0$. Besides of the assumption $(F_{4})$, $f$ also satisfies the following assumptions:\\ $(F_{1})$ $\vert f(x, u)\vert\leq C(1+\vert u\vert^{p-1})$ for some $C>0$ and $2<p<2_{\alpha}^{*}=\frac{2N}{N-2\alpha}$.\\
$(F_{2})$ $f(x, u)=o(u)$ uniformly in $x$ as $u\rightarrow 0$.\\
$(F_{3})$ $\frac{F(x, u)}{u^{2}}\rightarrow\infty$ uniformly in $x$ as $\vert u\vert\rightarrow \infty$, where $F(x, u)=\int_{0}^{u}f(x, s)ds$.\\

Now let us state the main results of this paper.\\
\begin{theorem} \label{thm1}
Assume that $(F_{1})-(F_{4})$ hold, $V$ and $f$ is 1-periodic in $x_{1},\ldots, x_{N}$, then problem (1.1) has a ground state solution.
\end{theorem}

Let $*$ denote the action of  $\mathbb{Z}^{N}$ on $H^{\alpha}(\rn)$ given by
\begin{equation}
\big(k*u\big)(x):=u(x-k),\quad\quad k\in \mathbb{Z}^{N}.
\end{equation}
If $V$ and $f$ is 1-periodic in $x_{1},\ldots, x_{N}$ and $u_{0}$ is a solution of problem $(1.1)$, then so is $k*u_{0}$ for all $k\in \mathbb{Z}^{N}$. Set
\begin{equation*}
{\mathcal O (u_{0}) } :=\{k*u_{0}: k\in\mathbb{Z}^{N}\}.
\end{equation*}
Two solutions $u_{1}$ and $u_{2}$ are said to be geometrically distinct if $\mathcal O (u_{1})$, $\mathcal O (u_{2})$ are disjoint and $u_{2}\neq [s_{u_{1}}, t_{u_{1}}]u_{1}$.

\begin{theorem} \label{thm1}
Under the assumptions of Theorem 1.1 and $f(x, u)$ is odd in $u$, there are the infi\-nitely many geometrically distinct solutions for problem (1.1).
\end{theorem}

\begin{theorem} \label{thm2}
Assume that $(F_{1})-(F_{4})$ and $\lim_{\vert x\vert\rightarrow \infty}V(x)=+\infty$ hold, then problem (1.1) has a ground state solution.
\end{theorem}

\begin{theorem} \label{thm3}
Assume that $f(x, u)=f(u)$ and $(F_{1})-(F_{4})$ hold, $\inf_{x\in \rn}V(x)\leq V(x)<\lim_{\vert x\vert\rightarrow \infty}V(x)=\sup_{x\in \rn}V(x)<+\infty$,  then problem (1.1) has a ground state solution.
\end{theorem}

Let $\mathcal{F}$ be the class of functions $h\in L^{\infty}(\rn)$ such that for every $\epsilon>0$, the set $\{x\in \rn: \vert h(x)\vert\geq\epsilon\}$
has finite Lebesgue measure.

\begin{theorem} \label{thm4}
Besides of $(F_{1})-(F_{4})$, $V$ and $f$ also satisfies the following assumptions:\\
\textbf{$(V_{1})$} There exists a functions $V_{p}\in C(\rn, \r)$, 1-periodic in $x_{1}$, $\ldots$, $x_{N}$, such that $(V-V_{p})\in \mathcal{F}$, and $V(x)\leq V_{p}(x)$ for all
$x\in \rn$.\\
$(F_{5})$ There exists a function $f_{p}\in C(\rn\times\r, \r)$, 1-periodic in $x_{1},\ldots, x_{N}$, such that\\
$(i)$ $\vert f(x, u)\vert\geq \vert f_{p}(x, u)\vert$, \, $(x, u)\in (\rn, \r)$;\\
$(ii)$ $\vert f(x, u)-f_{p}(x, u)\vert \leq \vert h(x)\vert(\vert u\vert+\vert u\vert^{p-1})$, $(x, u)\in (\rn, \r)$, $h\in \mathcal{F}$ and $p\in [2, 2_{\alpha}^{*})$;\\
$(iii)$ $u\mapsto\frac{f_{p}(x, u)}{\vert u\vert}$ is nondecreasing on $(-\infty, 0)$ and on $(0, \infty)$.\\
Then problem (1.1) has a ground state solution.
\end{theorem}

The rest of the paper is organized as follows. In Section 2, we present some necessary preliminary knowledge. In Section 3 we prove Theorem 1.1 and Theorem 1.2. In Section 4 we prove Theorem 1.3. In Section 5 Theorem 1.4 is proved and in the final section we prove Theorem 1.5.\\

\bigskip

\noindent \textbf{Notation.} $C, C_1, C_2$ etc. will denote positive constants whose exact values are inessential. $\langle .\, , . \rangle$ is the inner product in the Hilbert space $E$.

\section{Preliminaries} \label{prel}
For any $\alpha\in (0, 1)$, the fractional Laplacian $(-\Delta)^{\alpha} u$ of a function $u:\rn\rightarrow \r$, with sufficient decay, is defined by
\begin{equation*}  \label{1}
\mathcal{F}((-\Delta)^{\alpha} u)(\xi)=\vert \xi\vert^{2\alpha}\mathcal{F}(u)(\xi), \quad  \xi\in \rn,
\end{equation*}
where $\mathcal{F}$ denotes the Fourier transform, that is,
\begin{equation*}
\mathcal{F}(\phi)(\xi)=\frac{1}{(2\pi)^{\frac{N}{2}}}\irn  e^{-i\xi\cdot x}\phi(x)dx \equiv \widehat{\phi}(\xi),
\end{equation*}
for function $\phi$ in the Schwartz class. $(-\Delta)^{\alpha} u$ can also be computed by the following singular integral:\\
\begin{equation*}
(-\Delta)^{\alpha} u=c_{N, \alpha}P.V.\irn \frac{u(x)-u(y)}{\vert x-y\vert^{N+2\alpha}}dy,
\end{equation*}
here P.V. is the principal value and $c_{N, \alpha}$ ia a normalization constant.\\

The fractional Sobolev space $H^{\alpha}(\rn)$ is defined by
\begin{equation*} \label{space}
 H^\alpha(\rn)=\Big\{u\in L^{2}(\rn): \frac{\vert u(x)-u(y)\vert}{\vert x-y\vert^{\frac{N}{2}+\alpha}}\in L^{2}(\rn\times \rn)\Big\},
\end{equation*}
endowed with the norm
\begin{equation} \label{space}
 \Vert u\Vert_{H^\alpha(\rn)}=\Big(\irn u^{2}dx+\iint_{\rn\times\rn}\frac{\vert u(x)-u(y)\vert^{2}}{\vert x-y\vert^{N+2\alpha}}dxdy\Big)^{\frac{1}{2}},
\end{equation}
where the term
\begin{equation*} \label{space}
[u]_{H^\alpha(\rn)}=\Vert (-\Delta)^{\frac{\alpha}{2}} u\Vert_{L^{2}(\rn)}:=\Big(\iint_{\rn\times\rn}\frac{\vert u(x)-u(y)\vert^{2}}{\vert x-y\vert^{N+2\alpha}}dxdy\Big)^{\frac{1}{2}}
\end{equation*}
is the so-called \emph{Gagliardo semi-norm} of $u$.\\

For the basic properties of the fractional Sobolev space $H^{\alpha}(\rn)$, we refer to \cite{rDPV, rS1, rS2}.\\

\begin{proposition}[{\cite{rDPV, rS2}}] \label{subdiff}
Let $0<\alpha<1$ such that $2\alpha<N$. Then there exists a constant $C=C(N, \alpha)>0$ such that
\begin{equation*}
\Vert u\Vert_{L^{2_{\alpha}^{*}}(\rn)}\leq C\Vert u\Vert_{H^{\alpha}(\rn)}
\end{equation*}
for every $u\in H^{\alpha}(\rn)$, where $2_{\alpha}^{*}=\frac{2N}{N-2\alpha}$ is the fractional critical exponent. Moreover, the embedding
$H^{\alpha}(\rn)\subset L^{q}(\rn)$ is continuous for any $q\in [2, 2_{\alpha}^{*}]$, and is locally compact
whenever $q\in [2, 2_{\alpha}^{*})$.
\end{proposition}

\begin{proposition}[{\cite{rS2}}] \label{subdiff}
Assume that $\{u_{n}\}$ is bounded in $H^{\alpha}(\rn)$ and it satisfies
\begin{equation*}
\underset{n\rightarrow\infty}{\lim}\underset{y\in \rn}{\sup}\int_{B_{R}(y)}\vert u_{n}(x)\vert^{2}dx=0,
\end{equation*}
where $R>0$. Then $u_{n}\rightarrow 0$ in $L^{q}(\rn)$ for every $2<q<2_{\alpha}^{*}$.
\end{proposition}

Let $E$ be the Hilbert subspace of $u\in H^{\alpha}(\rn)$ under the norm
\begin{equation*} \label{space}
 \Vert u\Vert=\Big(\irn  V(x)u^{2}dx+\iint_{\rn\times\rn}\frac{\vert u(x)-u(y)\vert^{2}}{\vert x-y\vert^{N+2\alpha}}dxdy\Big)^{\frac{1}{2}}.
\end{equation*}
It is clear that this norm is equivalent to (2.1) if $V\in L^{\infty}(\rn)$ and $E\subset H^{\alpha}(\rn) $ if $\lim_{\vert x\vert\rightarrow \infty}V(x)=+\infty$. Moreover, by \cite{rPXZ}, if $V$ is coercive, for any $q\in [2, 2_{\alpha}^{*})$,
the embedding $E\hookrightarrow L^{q}(\rn)$ is compact.\\

The energy functional $J$ on $E$ associated with problem (1.1) is
\begin{equation} \label{fcl}
J(u):= \frac{1}{2}\irn  \Big(\vert (-\Delta)^{\frac{\alpha}{2}} u\vert^{2}+V(x)u^{2}\Big)dx-\irn F(x, u)\,dx.
\end{equation}
Under the assumptions $(F_{1})-(F_{2})$, $J\in C^{1}(E, \r)$ and its critical points are solutions of problem (1.1).
Define Nehari manifold associated to the functional $J$,
\begin{equation}
\mathcal{N}=\Big\{u\in E\setminus\{0\}: \irn\Big(\vert  (-\Delta)^{\frac{\alpha}{2}} u\vert^{2}+V(x)u^{2}\Big)dx=\irn f(x, u)u\,dx\Big\}.
\end{equation}
It is easy to know that $\mathcal{N}$ is closed and if $u$ is a nontrivial critical point of $J$, then $u\in\mathcal{N}$.

\section{Proofs of Theorem 1.1 and Theorem 1.2} \label{pf1}
First, by $(F_{1})$ and $(F_{2})$, for any $\epsilon>0$ there exists $C_{\epsilon}>0$ such that
\begin{equation}
\vert f(x, u)\vert\leq \epsilon\vert u\vert+C_{\epsilon}\vert u\vert^{p-1} \quad\quad \text{for all}\,  x\in \rn\, \text{and}\, u\in \r.
\end{equation}
It is also easy to see  from $(F_{3})$ that $f(x, tu)\neq 0$ if $u\neq 0$ and $t>0$ large enough. Moreover, for any $u\neq 0$, by $(F_{2})$ and $(F_{4})$,  one has
\begin{equation}
\frac{1}{2}f(x, u)u\geq F(x, u)\geq 0.
\end{equation}

\begin{lemma}
(i)For any $u\in E\setminus \{0\}$,  there exists  $t_{u}=t(u)>0$ such that $t_{u}u\in \mathcal{N}$ and $J(t_{u}u)=\underset{t\geq 0}{\max}J(tu)$. \\
(ii) There exists $\alpha_{0}>0$ such that $\Vert \omega\Vert\geq \alpha_{0}$ for all $\omega\in\mathcal{N}$.\\
(iii) If $\omega\in \mathcal{N}$, then there exist $0<s_{\omega}\leq 1\leq t_{\omega}$ such that $[s_{\omega}, t_{\omega}]\omega\subset \mathcal{N}$. Moreover, $J(s\omega)=J(\omega)$, $J'(s\omega)=sJ'(\omega)$ for all $s\in [s_{\omega}, t_{\omega}]$
and $J(z)<J(\omega)$ for all other $z\in \r^{+}u\backslash [s_{\omega}, t_{\omega}]\omega$.\\
(iv)There exists $\rho>0$, such that $c:=\underset{\mathcal{N}}{\inf}J\geq \underset{S_{\rho}}{\inf}J>0$, where $S_{\rho}:=\{u\in E: \Vert u\Vert=\rho\}$.\\
(v)If $\mathcal{V}\subset E\setminus\{0\}$ is a compact subset, then there exists $R>0$ such that $J\leq 0$ on $\r^{+}u\setminus B_{R}(0)$ for every $u\in\mathcal{V}$.\\
(vi) $J$ is coercive on $E$, i.e., $J(u)\rightarrow \infty$, as $\Vert u\Vert\rightarrow\infty$.

\end{lemma}

\begin{proof}
(i) Let $u\in E\setminus \{0\}$ be fixed and define the function $g(t):=J(tu)$ on $[0, \infty)$. By $(3.1)$ and Proposition 2.1, for $\epsilon$ small enough we obtain that
\begin{align*}
g(t)&=\frac{t^{2}}{2}\Vert u\Vert^{2}-\int_{\mathbb{R}^{N}}F(x, tu)dx\nonumber\\
&\geq \frac{t^{2}}{2}\Vert u\Vert^{2}-\epsilon t^{2}\int_{\mathbb{R}^{N}}u^{2}dx-C_{\epsilon} t^{p}\int_{\mathbb{R}^{N}}\vert u\vert^{p} dx.\nonumber\\
&\geq \frac{t^{2}}{4}\Vert u\Vert^{2}-C_{1}t^{p}\Vert u\Vert^{p}.
\end{align*}
Since $u\neq 0$ and $2<p<2_{\alpha}^{*}$, $g(t)>0$ whenever $t>0$ is small enough.
On the other hand, sicne $u(x)\neq 0$, $\vert tu(x)\vert\rightarrow \infty$ as $t\rightarrow +\infty$. By $(F_{3})$, we have
\begin{align*}
g(t)&=\frac{t^{2}}{2}\Vert u\Vert^{2}-\int_{\mathbb{R}^{N}}F(x, tu)dx\nonumber\\
&=\frac{t^{2}}{2}\Vert u\Vert^{2}-t^{2}\int_{u\neq 0}\frac{F(x, tu)}{(tu)^{2}}u^{2} dx\nonumber
\end{align*}
which yields that $g(t)\rightarrow -\infty$ as $t\rightarrow \infty$. Therefore $\max_{t\geq 0}g(t)$ is achieved at some $t_{u}=t(u)>0$, and $g'(t_{u})=J'(t_{u}u)u=0$, $t_{u}u\in \mathcal{N}$. \\

(ii) For any $\omega\in\mathcal{N}$, by (3.1), we have
\begin{align*}
0=\langle J'(\omega), \omega\rangle &=\Vert \omega\Vert^{2}-\irn f(x, \omega)\omega\,dx\\
&\geq \Vert \omega\Vert^{2}-\epsilon\irn \omega^{2}dx-C_{\epsilon}\irn \vert \omega\vert^{p}dx\\
&\geq \frac{1}{2}\Vert \omega\Vert^{2}-C_{1}\Vert \omega\Vert^{p}.
\end{align*}
Since $2<p<2_{\alpha}^{*}$,  there exists $\alpha_{0}>0$ such that $\Vert \omega\Vert\geq \alpha_{0}$.\\

(iii) For any  $\omega\in\mathcal{N}$, for $\epsilon$ small, we have
\begin{align}
\irn(\vert  (-\Delta)^{\frac{\alpha}{2}} \omega\vert^{2}+V(x)\omega^{2})dx=\irn f(x, \omega)\omega\,dx.
\end{align}
Moreover, set $s>0$ such that $s\omega\in\mathcal{N}$, that is
\begin{align*}
0=\langle J'(s\omega), s\omega\rangle =s^{2}\irn(\vert  (-\Delta)^{\frac{\alpha}{2}} \omega\vert^{2}+V(x)\omega^{2})dx-\irn f(x, s\omega)s\omega\,dx,
\end{align*}
and
\begin{align}
\irn(\vert  (-\Delta)^{\frac{\alpha}{2}} \omega\vert^{2}+V(x)\omega^{2})dx=\irn \frac{ f(x, s\omega)}{s\omega}\omega^{2}\,dx.
\end{align}
Combing  (3.3) and (3.4), we have
\begin{align}
\irn \Big(\frac{ f(x, s\omega)}{s\omega}-\frac{f(x, \omega)}{\omega}\Big)\omega^{2}\,dx=0
\end{align}
By $(F_{4})$, there exist $s_{\omega}$, $t_{\omega}$ such that $s_{\omega}\in (0, 1]$, $t_{\omega}\geq 1$, and for any $s\in [s_{\omega}, t_{\omega}]$,
$f(x,s\omega)\omega=f(x, \omega)s\omega$ and $s\omega\in \mathcal{N}$. Moreover, it is an immediate consequence from (i) that $J(s\omega)=J(\omega)$, $J'(s\omega)=sJ'(\omega)$ for all $s\in [s_{\omega}, t_{\omega}]$
and $J(z)<J(\omega)$ for all other $z\in \r^{+}u\backslash [s_{\omega}, t_{\omega}]\omega$. \\

(iv) According to $(3.1)$, if $\Vert u\Vert\leq 1$ and $\epsilon$ small, we have
\begin{align*}
J(u)&=\frac{1}{2}\Vert u\Vert^{2}-\irn F(x, u)\,dx\\
&\geq \frac{1}{2}\Vert u\Vert^{2}-C_{1}\epsilon\Vert u\Vert^{2}-C_{2}C_{\epsilon}\Vert u\Vert^{p}\nonumber\\
&\geq \frac{1}{4}\Vert u\Vert^{2}-C_{2}C_{\epsilon}\Vert u\Vert^{p}.
\end{align*}
Since $p>2$, $\underset{u\in S_{\rho}}{\inf} J(u)>0$  if  $\Vert u\Vert=\rho$ and $\rho$ small enough. Moreover,
according to (i) and (ii), for every $\omega\in \mathcal{N}$, there exists  $t>0$, such that $t\omega\in S_{\rho}$ and $J(\omega)\geq J(s\omega)$, so $c:=\underset{\mathcal{N}}{\inf}J\geq \underset{S_{\rho}}{\inf}J>0$.\\

(v)Without loss of generality, we may assume that $\Vert u\Vert=1$ for every $u\in \mathcal{V}$. Arguing by contraction, suppose that there exist $\{u_{n}\}\subset\mathcal{V}$ and $t_{n}u_{n}\in \r^{+}u_{n}$, $n\in N$ such that $J(t_{n}u_{n})\geq 0$ for all $n\in N$ and $t_{n}\rightarrow\infty$ as $n\rightarrow \infty$. Up to a subsequence, we may assume that $u_{n}\rightarrow u\in S:=\{u\in E: \Vert u\Vert=1\}$. It is clear that
\begin{align*}
0\leq \frac{J(t_{n}u_{n})}{t_{n}^{2}\Vert u_{n}\Vert^{2} }=\frac{1}{2}-\irn \frac{F(x, t_{n}u_{n})}{t_{n}^{2}u_{n}^{2}}u_{n}^{2}\,dx.\\
\end{align*}
By $(F_{3})$ and Fatou's lemma, one has
\begin{align*}
\irn \frac{F(x, t_{n}u_{n})}{t_{n}^{2}u_{n}^{2}}u_{n}^{2}\,dx\rightarrow +\infty,
\end{align*}
this yields a contradiction.\\

(vi) Arguing by contraction, suppose there exists a sequence $\{\omega_{n}\}\subset\mathcal{N}$ such that $\Vert \omega_{n}\Vert\rightarrow \infty$ and
$J(\omega_{n})\leq d$ for some $d>0$. Let $v_{n}=\frac{\omega_{n}}{\Vert \omega_{n}\Vert }$, then $v_{n}\rightharpoonup v$ in $E$ and $v_{n}(x)\rightarrow v$ a.e. $x\in\rn$,
up to a subsequence. For some $R>0$, choose $y_{n}\in\rn$ satisfy \\
\begin{align*}
\int_{B_{R}(y_{n})}v_{n}^{2}dx=\underset{y\in\rn}{\sup}\int_{B_{R}(y)}v_{n}^{2}dx.
\end{align*}
Suppose $\int_{B_{R}(y_{n})}v_{n}^{2}dx\rightarrow 0$ as $n\rightarrow\infty$, then $v_{n}\rightarrow 0$ in $L^{p}(\rn)$ by Proposition 2.2. Moreover, by (3.1),
$\irn F(x, sv_{n})dx\rightarrow 0$ for all $s\in\r$ and therefore
\begin{align*}
d\geq J(\omega_{n})\geq J(sv_{n})=\frac{s^{2}}{2}-\irn F(x, sv_{n})dx\rightarrow \frac{s^{2}}{2},
\end{align*}
this yields a contraction if $s>\sqrt{2d}$. So there exists $\delta>0$ such that
\begin{align}
\int_{B_{R}(y_{n})}v_{n}^{2}dx\geq \delta.
\end{align}
Since $J$ is invariant under translations of the form $u\mapsto u(\cdot-k)$ with $k\in Z^{N}$, we may assume that
 $\{y_{n}\}$ is bounded in $\rn$. By (3.6) and Fatou's lemma, we have $v\neq 0$. Moreover, by (iv) and $(F_{3})$, we obtain
\begin{align*}
0\leq \frac{J(\omega_{n})}{\Vert \omega_{n}\Vert^{2} }=\frac{1}{2}-\irn \frac{ F(x, \omega_{n})}{\Vert \omega_{n}\Vert^{2}}dx=\frac{1}{2}-\irn \frac{ F(x,  \omega_{n}) }{\omega_{n}^{2}}v_{n}^{2}dx\rightarrow -\infty.
\end{align*}
This yields a contradiction. Thus, $J$ is coercive on $\mathcal{N}$.
\end{proof}
\begin{remark} \label{cptsupp}
\emph{
In the proof of Lemma 3.1(iv), to show that $v\neq 0$, we use the assumption $f$ and $V$ are 1-periodic in $x_{1},\ldots, x_{N}$.
So for the coercive potential case, the bounded potential well case and the $V$ and $f$ are asymptotically periodic in $x$ case, we need to adapt the proof.}
\end{remark}

In virtue of Lemma 3.1 (iii), for any $u\in E\setminus\{0\}$, there exist $\omega\in\mathcal{N}$ and $0<s_{\omega}\leq 1\leq t_{\omega}$ such that
\begin{equation*}
m(u):=[s_{\omega}, t_{\omega}]\omega=\r^{+}u\cap\mathcal{N},
\end{equation*}
where $m:E\setminus \{0\}\rightarrow E$ is a multiplevalued map. Define
\begin{equation*}
\widehat{\Psi}(u):=J(m(u))=\underset{v\in \r^{+}u}{\max}J(v).
\end{equation*}
This is a single-valued map since $J$ is constant on $\r^{+}u\cap\mathcal{N}$. If $(F'_{4})$ holds instead of $(F_{4})$, by the same proof as in \cite{rSW},
for any $u\in E\setminus\{0\}$, there exists a unique positive number $t>0$ such that $J(tu)=\underset{t>0}{\max}J(tu)$, so $\widehat{\Psi}\in C^{1}(E\setminus\{0\}, \r)$. But under our assumptions, $u\mapsto m(u)$ may not be single value, thus $\widehat{\Psi}$ may not be $C^{1}$ in
$E\setminus\{0\}$ and the critical points theory for smooth functionals does not work, we need the nonsmooth methods in \cite{rC}.

\begin{proposition}
$\widehat{\Psi}:E\setminus\{0\}\rightarrow R$ is a locally Lipschitz continuous map.
\end{proposition}
\begin{proof}
If $u_{0}\in E\setminus\{0\}$, then there exist a neighborhood $U\subset E\setminus\{0\}$ of $u_{0}$ and $R>0$ such that
$J(\omega)\leq 0$ for all $u\in U$ and $\omega\in \r^{+}u$, $\Vert \omega\Vert\geq R$. Arguing by contraction, suppose that
there exist sequence $\{u_{n}\}$, $\{\omega_{n}\}$ such that $u_{n}\rightarrow u_{0}$,  $\omega_{n}\in \r^{+}u_{n}$, $J(\omega_{n})>0$
and $\Vert \omega_{n}\Vert\rightarrow\infty$. Since $u_{0}$, $u_{1}$, $u_{2}, \ldots$ is a compact subset, by Lemma 3.1(v), we have
$J(\omega)\leq 0$ for some $R>0$ and all $\omega\in \r^{+}u_{n}, n=0, 1, 2, \ldots$, $\Vert \omega\Vert\geq R$, this is a contraction.\\
Let $t_{1}u_{1}\in m(u_{1})$, $t_{2}u_{2}\in m(u_{2})$, where $u_{1}, u_{2}\in U$, then $\Vert m(u_{1})\Vert$, $\Vert m(u_{2})\Vert\leq R$.
By the maximality property of $m(u)$ and the mean value theorem,
\begin{align*}
\widehat{\Psi}(u_{1})-\widehat{\Psi}(u_{2})&=J(t_{1}u_{1})-J(t_{2}u_{2})\leq J(t_{1}u_{1})-J(t_{1}u_{2})\\
&\leq t_{1}\underset{s\in [0,1]}{\sup}\Vert J'(t_{1}(su_{1}+(1-s)u_{2}))\Vert u_{1}-u_{2}\Vert\\
&\leq C\Vert u_{1}-u_{2}\Vert,
\end{align*}
where $C$ depends on $R$ but independs on the particular choice of in  $m(u_{1})$, $ m(u_{2})$. Similarly the above inequality, we also have
\begin{align*}
\widehat{\Psi}(u_{2})-\widehat{\Psi}(u_{1})\leq C\Vert u_{1}-u_{2}\Vert.
\end{align*}
This completes the proof.
\end{proof}

For each $v\in E$, the generalized directional derivative $\widehat{\Psi}^{\circ}(u;v)$ in the direction $v$ is defined by
\begin{align*}
\widehat{\Psi}^{\circ}(u;v)=\underset{\underset{t\downarrow 0}{h\rightarrow 0}}{\lim\sup}\frac{\widehat{\Psi}(u+h+tv)-\widehat{\Psi}(u+h)}{t}.
\end{align*}
The function $v\mapsto \widehat{\Psi}^{\circ}(u;v)$ is subadditive and positively homogeneous, and then is convex. The generalized gradient of $\widehat{\Psi}$
at $u$, denoted $\partial\widehat{\Psi}(u)$, is defined to be subdifferential of the convex function $\widehat{\Psi}^{\circ}(u;v)$ at $\theta$, that is, $\omega\in \partial\widehat{\Psi}(u)\subset E$ if and only if for all $v\in E$,
\begin{align*}
\widehat{\Psi}^{\circ}(u;v)\geq \langle\omega, v\rangle.
\end{align*}
If $0\in \partial\widehat{\Psi}(u)$, i.e. $\widehat{\Psi}^{\circ}(u;v)\geq 0$, for all $v\in E$,
we call $u$ is a critical point of $\widehat{\Psi}$. We call a sequence $\{u_{n}\}$ is a Palais-Smale sequence for $\widehat{\Psi}$(PS-sequence for short)
if $\widehat{\Psi}(u_{n})$ is bounded and there exist $\omega_{n}\in \partial\widehat{\Psi}(u_{n})$ such that $\omega_{n}\rightarrow 0$. The functional
$\widehat{\Psi}$ satisfies the PS-condition if each PS-sequence has a convergent subsequence.\\
We shall use some notations
\begin{align*}
S:=\{u\in E: \Vert u\Vert=1\},\quad T_{u}S:=\{v\in E: \langle u, v\rangle=0\},\quad \Psi=\widehat{\Psi}\mid_{S},\\
\Psi^{d}:=\{u\in S:\Psi(u)\leq d \}, \quad \Psi_{c}:=\{u\in S:\Psi(u)\geq c\},\quad \Psi_{c}^{d}:=\Psi_{c}\cap\Psi^{d},\\
K:=\{u\in S:0\in \partial \widehat{\Psi}(u)\},\quad K_{c}:=\Psi_{c}^{c}\cap K,\quad \partial \Psi(u):=\partial \widehat{\Psi}(u),\, \text{if}\, u\in S.
\end{align*}
\begin{proposition}
(i)$u\in S$ is a critical point of $\widehat{\Psi}$ if and only if $m(u)$ consists of critical points of $J$. The corresponding critical values coincide.\\
(ii)$\{u_{n}\}\subset S$ is a PS-sequence for $\widehat{\Psi}$ if and only if there exist $\omega_{n}\in m(u_{n})$ such that $\{\omega_{n}\}$ is a PS-sequence for
$J$.
\end{proposition}
\begin{proof}
(i) In fact, we need to show that for $u\in S$, $\widehat{\Psi}^{\circ}(u;v)\geq 0$ for any $v\in E$ if and only if $m(u)$ consists of critical points of $J$.
It is clear that $E=\r u\oplus T_{u}S$, and by the maximizing property of $m(u)$, $J'(\omega)v=0$ for all $\omega\in m(u)$ and $v\in \r u$. Fixed $s\in \r$, since  $\widehat{\Psi}$ is locally Lipschitz continuous and $\widehat{\Psi}(u)=\widehat{\Psi}(\sigma u)$ for all $\sigma>0$,  we have
\begin{align*}
\vert\widehat{\Psi}(u+h+t(su))-\widehat{\Psi}(u+h)\vert=\vert\widehat{\Psi}((1+ts)u+h)-\widehat{\Psi}((1+ts)(u+h))\vert\leq Ct\vert s\vert\Vert h\Vert
\end{align*}
for $t>0$ and $\Vert h\Vert$ small enough. Thus $\widehat{\Psi}^{\circ}(u;su)=0$ for all $s\in \r$.\\
Let $s_{u}>0$ such that  $s_{u}u\in m(u)$, by the maximizing property of $m(u)$ and the mean value theorem,
\begin{align*}
\frac{\widehat{\Psi}(u+h+tv)-\widehat{\Psi}(u+h)}{t}&=\frac{J(s_{u+h+tv}(u+h+tv))-J(s_{u+h}(u+h))}{t}\\
&\leq \frac{J(s_{u+h+tv}(u+h+tv))-J(s_{u+h+tv}(u+h))}{t}\\
&=s_{u+h+tv}J'(s_{u+h+tv}(u+h+\theta tv))v
\end{align*}
where $t>0$ and $\theta\in (0, 1)$.  Since $\mathcal{N}$ is bounded away from 0 and $J$ is coercive on $\mathcal{N}$,
thus $s_{u+h+tv}$ is bounded. Letting $h\rightarrow 0$ and $t\downarrow 0$ via subsequences, we have
\begin{align}
\widehat{\Psi}^{\circ}(u;v)\leq sJ'(su)v,
\end{align}
where $s_{n}:=s_{u+h_{n}+t_{n}v}\rightarrow s>0$. Moreover, since $\mathcal{N}$ is closed and
\begin{align*}
\widehat{\Psi}(u)&=\underset{n\rightarrow\infty}{\lim}\widehat{\Psi}(u+h_{n}+t_{n}v)=\underset{n\rightarrow\infty}{\lim}J(s_{n}(u+h_{n}+t_{n}v))\\
&=\underset{n\rightarrow\infty}{\lim}\Big(\frac{s_{n}^{2}}{2}\Vert u+h_{n}+t_{n}v\Vert^{2}-\irn F(x, s_{n}(u+h_{n}+t_{n}v))dx\Big)\\
&=\frac{s^{2}}{2}\Vert u\Vert^{2}-\irn F(x, su)dx,
\end{align*}
so $su\in \mathcal{N}$. From Lemma 3.1(iii), it is possible that $\r^{+}u\cap\mathcal{N}$ is a line segment, not a point, hence s may be different for different $v$. We set $s_{1}, s_{2}$ correspond to $v_{1}$ and $v_{2}$, by Lemma 3.1(iii), we have $s_{1}u=\tau (s_{2}u)$ and $J'(s_{1}u)v=\tau J'(s_{2}u)v$ for some $\tau>0$. From this
and (3.7), for $y\in\partial\Psi(u)$, one has
\begin{equation}
\langle y, v\rangle\leq \widehat{\Psi}^{\circ}(u;v)\leq\tau(v)J'(su)v,
\end{equation}
where $\tau$ is bounded and bounded away from 0(by constants independent of $v$). It follows that  $u$ is a critical point of $\widehat{\Psi}$ if and only if
$m(u)$ consists of critical points of $J$.\\

(ii)We take $y_{n}\in\partial\Psi(u_{n})$ and $\omega_{n}\in m(u_{n})$. Since $J$ is coercive on $\mathcal{N}$ and  $J(m(u_{n}))$ is bounded, the sequence
$\{m(u_{n})\}$ is bounded. As in (3.8), we have
\begin{equation}
\langle y_{n}, v\rangle\leq \widehat{\Psi}^{\circ}(u_{n};v)\leq\tau_{n}(v)J'(\omega_{n})v,
\end{equation}
where $\tau_{n}$ is bounded and bounded away from 0 because so is $m(u_{n})$. We complete the proof.
\end{proof}
\begin{remark} \label{cptsupp}
\emph{(1) Because of $\widehat{\Psi}^{\circ}(u; su)=0$ for all $s\in \r$, so $\partial\Psi(u)\subset T_{u}S$.}\\
\emph{(2) If $\omega_{n}\in m(u_{n})$ is a PS-sequence of $J$, then so is any sequence $\omega'_{n}\in m(u_{n})$.}\\
\end{remark}

The pseudo-gradient vector field $H:S\setminus K\rightarrow TS$ for $\Psi$ be very important. For $u\in S$, we define
\begin{equation}
\partial^{-}\Psi(u):=\{\gamma\in \partial\Psi(u): \Vert \gamma\Vert=\underset{a\in \partial\Psi(u)}{\min}\Vert a\Vert\}
\end{equation}
and
\begin{equation*}
\mu(u):=\underset{\beta\in S}{\inf}\{\Vert \partial^{-}\Psi(\beta)\Vert +\Vert u-\beta\Vert\}.
\end{equation*}
Because $\partial\Psi(u)$ is a closed and convex set, from \cite{rPKS}, it follows  that $\gamma$ in (2.1) exists and is unique, so we have
\begin{equation*}
K=\{u\in S: \partial^{-}\Psi(u)=0\}.
\end{equation*}
By \cite{rCh}, the map $u\mapsto \Vert \partial^{-}\Psi(u)\Vert$ is lower semicontinuous but not continuous in general. To regularize $\Vert \partial^{-}\Psi(u)\Vert$, the function $\mu$ be introduced.
\begin{lemma}
The function $\mu$ is continuous and $u\in K$ if and only if $\mu(u)=0$.
\end{lemma}
\begin{proof}
Let $u, v, \beta\in S$, by the definition of $\mu$, we have
\begin{equation*}
\mu(u)\leq \Vert \partial^{-}\Psi(\beta)\Vert +\Vert u-\beta\Vert\leq \Vert \partial^{-}\Psi(\beta)\Vert +\Vert v-\beta\Vert+\Vert u-v\Vert.
\end{equation*}
So
\begin{align*}
\mu(u)&\leq \underset{\beta\in S}{\inf}\{\Vert \partial^{-}\Psi(\beta)\Vert +\Vert v-\beta\Vert\}+\Vert u-v\Vert\\
&=\mu(v)+\Vert u-v\Vert.
\end{align*}
Similarity, we have
\begin{align*}
\mu(v)-\mu(v)\leq\Vert u-v\Vert,
\end{align*}
Hence $\mu$ is Lipschitz continuous and is also continuous.\\
Since $0\leq \mu(u)\leq\Vert \partial^{-}\Psi(\beta)\Vert$, it is easy to see that $\mu(u)=0$ if $u\in K$. Now suppose $\mu(u)=0$. In virtue of the definition
of $\mu(u)$, there exist $\beta_{n}\subset S$ such that $\partial^{-}\Psi(\beta_{n})\rightarrow 0$ and $\beta_{n}\rightarrow u$. Moreover, by the map $u\mapsto \Vert \partial^{-}\Psi(u)\Vert$ is lower semicontinuous, so $u\in K$.

\end{proof}
\begin{proposition}
There exists a locally Lipschitz continuous vector field $H:S\setminus K\rightarrow TS$ with $\Vert H(u)\Vert\leq 1$ and $\inf\{\langle \gamma, H(u)\rangle: \gamma\in \partial\Psi(u)\}>\frac{1}{2}\mu(u)$ for all $u\in S\setminus K$. If $J$ is even, then $H$ may be chosen to be odd.
\end{proposition}
This follows by an easy inspection of the proof of Proposition 2.10 in \cite{rPKS}.

\begin{proof}[\textbf{Proof of Theorem 1.1}]
Since $c:=\underset{u\in S}{\inf}\Psi(u)=\underset{\omega\in \mathcal{N}}{\inf}J(\omega)>0$ from Lemma 3.1(iv). By Ekeland's variational principle, there is a sequence $\{u_{n}\}\subset S$ such that $\Psi(u_{n})\rightarrow c$ and
\begin{equation}
\Psi(v)\geq \Psi(u_{n})-\frac{1}{n}\Vert v-u_{n}\Vert\quad \text{for all}\, v\in S.
\end{equation}
For a given $v\in T_{u_{n}}S$, let $z_{n}(t)=\frac{u_{n}+tv}{\Vert u_{n}+tv\Vert}$. It is clear that $\Vert u_{n}+tv\Vert-1=O(t^{2})$ as $t\rightarrow 0$
and $\widehat{\Psi}(u_{n}+tv)=\Psi(z_{n}(t))$. From (3.11), we have
\begin{align*}
\widehat{\Psi}^{\circ}(u_{n};v)\geq \underset{t\downarrow 0}{\lim\sup}\frac{\widehat{\Psi}(u_{n}+tv)-\widehat{\Psi}(u_{n})}{t}=\underset{t\downarrow 0}{\lim\sup}\frac{\Psi(z_{n}(t))-\Psi(u_{n})}{t}\geq -\frac{1}{n}\Vert v\Vert.
\end{align*}
Since $J$ is coercive on $\mathcal{N}$,  $\{m(u_{n})\}$ is bounded. Moreover, by (3.9), one has
\begin{align*}
-\frac{1}{n}\Vert v\Vert\leq \widehat{\Psi}^{\circ}(u;v)\leq \tau_{n}(v)J'(\omega_{n})v,
\end{align*}
where $\omega_{n}\in m(u_{n})\subset \mathcal{N}$ and $\tau_{n}$ is bounded and bounded away from 0. Since for any $v\in \r\omega_{n}$, $J'(\omega_{n})v=0$,
 $\{\omega_{n}\}$ is a bounded PS-sequence of $J$. Up to a subsequence,  $\omega_{n}\rightharpoonup u$, $\omega_{n}\rightarrow u$ in $L_{loc}^{2}(\rn)$
 and $\omega_{n}\rightarrow u$ in a.e. $x\in \rn$.
 If $\omega_{n}\rightarrow 0$ in $L^{p}(\rn)$, then $\irn F(x, \omega_{n})dx\rightarrow 0$ and $\irn f(x, \omega_{n})\omega_{n}dx\rightarrow 0$ as $n\rightarrow\infty$, this implies that $\Vert \omega_{n}\Vert\rightarrow 0$ as $n\rightarrow\infty$ and this contradicts with Lemma 3.1(ii), so  $\omega_{n}\not\rightarrow 0$ in $L^{p}(\rn)$, by Proposition 2.2, for some $R>0$ and $\delta>0$, there exist $\{y_{n}\}$ such that
 \begin{align*}
\int_{B_{R}(y_{n})}\omega_{n}^{2}dx\geq \delta.
\end{align*}
 Since $J$ and $\mathcal{N}$ are invariant under translations of the form $\omega\mapsto \omega(\cdot-k)$ with $k\in Z^{N}$, we may assume that
 $\{y_{n}\}$ is bounded in $\rn$. By Fatou's lemma, we know that $u\neq 0$. Now we show that $u$ is a ground state solution. By (3.2) and Fatou's lemma
 \begin{align*}
c:=\underset{n\rightarrow \infty}{\lim}J(\omega_{n})&=\underset{n\rightarrow \infty}{\lim}\Big(J(\omega_{n})-\frac{1}{2}J'(\omega_{n})\omega_{n}\Big)\\
&=\underset{n\rightarrow \infty}{\lim}\irn \Big(\frac{1}{2}f(x, \omega_{n})\omega_{n}-F(x, \omega_{n})\Big)dx\\
&\geq \irn \Big(\frac{1}{2}f(x, u)u-F(x, u)\Big)dx\\
&=J(u)-\frac{1}{2}J'(u)u=J(u)\geq c.
\end{align*}
The proof is completed.
\end{proof}
\begin{remark} \label{cptsupp}
\emph{
Under the assumptions of Theorem 1.1, if $f(x, u)\geq 0$, $u\geq0$, $f(x, u)=0$, $u\leq 0$, then we may obtain a nonnegative ground state solution. Put $u^{+}:=\max\{u, 0\}$. Noting that the conclusion of Theorem 1.1 holds for the functional}
\begin{equation*} \label{fcl}
J^{+}(u):= \frac{1}{2}\irn\Big(\vert  (-\Delta)^{\frac{\alpha}{2}} u\vert^{2}+V(x)u^{2}\Big)dx-\irn F(x, u^{+})\,dx.
\end{equation*}
\emph{
So we get a ground state solution $u$ of the equation}
\begin{equation*}  \label{1}
(-\Delta)^{\alpha} u+V(x)u =f(x, u^{+}), \quad  x\in \rn.
\end{equation*}
\emph{Using $u^{-}:=\min\{u, 0\}$ as a test function in above equation, and integrating by parts, we obtain}
\begin{equation*} \label{fcl}
\irn (-\Delta)^{\alpha} u\cdot u^{-}dx=-\irn V(x)(u^{-})^{2}dx\leq 0.
\end{equation*}
\emph{But we know that}
\begin{align*}
\irn (-\Delta)^{\alpha} u\cdot u^{-}dx&=\iint_{\rn\times\rn}\frac{(u(x)-u(y))(u^{-}(x)-u^{-}(y))}{\vert x-y\vert^{N+2\alpha}}dxdy\\
&\geq \iint_{\{u>0\}\times \{u<0\}}\frac{(u(x)-u(y))(-u^{-}(y))}{\vert x-y\vert^{N+2\alpha}}dxdy\\
& \quad +\iint_{\{u<0\}\times \{u<0\}}\frac{(u^{-}(x)-u^{-}(y))^{2}}{\vert x-y\vert^{N+2\alpha}}dxdy\\
&\quad + \iint_{\{u<0\}\times \{u>0\}}\frac{(u(x)-u(y))u^{-}(x)}{\vert x-y\vert^{N+2\alpha}}dxdy\geq 0.
\end{align*}
\emph{Thus $u^{-}=0$ and $u\geq 0$ is a ground state solution of problem (1.1).}
\end{remark}

Now we assume that $f(x, u)$ is odd in $u$. To prove the existence of infinitely many geometrically distinct solutions, we assume the contrary. Since for each
$[s_{\omega}, t_{\omega}]\omega\subset\mathcal{N}$ there corresponds a unique point $u\in S$. Assume that $\mathcal{F}$ is a finite set and choose a subset $\mathcal{F}$ of $K$ such that $-\mathcal{F}=\mathcal{F}$ and each orbit $\mathcal{O}(\omega)$ has a unique representative in $\mathcal{F}$.
\begin{lemma}
\textit{The mapping} $m^{-1}: \mathcal{N}\rightarrow S$ \textit{is Lipschitz continuous.}
\end{lemma}

\begin{lemma}
$\kappa:=\inf\{\Vert v-w\Vert: v, w\in K, v\neq w\}>0$\textit{.}
\end{lemma}

The proofs of the above two lemmas are similar with Lemma 2.11 and Lemma 2.13 in \cite{rSW}, so we omit it here.

\begin{lemma}[{\cite{rSW}}]
Let $d\geq c$. If $\{(v_{n}^{1}\}$, $\{v_{n}^{2}\}\subset \Psi^{d}$ are two Palais-Smale sequences for
$\Psi$, then either $\Vert v_{n}^{1}-v_{n}^{2}\Vert\rightarrow 0$ as $n\rightarrow \infty$ or $\lim\sup_{n\rightarrow \infty}\Vert v_{n}^{1}-v_{n}^{2}\Vert\geq \rho(d)>0$, where $\rho(d)$ depends on $d$ but not on the particular choice of PS-sequences in $\Psi^{d}$.
\end{lemma}

Let $H$ be the pseudo-gradient vector field in Proposition and  $\eta: \rightarrow S\setminus K$ be the flow defined by
\begin{align}
 \left\{
\begin{aligned}
&\frac{d}{dt}\eta(t, w)=-H(\eta(t, w)),\\
&\eta(0, w)=w,
\end{aligned}
\right.
\end{align}
where $\mathcal{G}:=\{(t, w): w\in S\setminus K, T^{-}(w)<t<T^{+}(w)\}$ and $(T^{-}(w), T^{+}(w))$ is the maximal existence time for the trajectory $t\mapsto \eta(t, w)$ which passing through $\omega$ at $t=0$.
Note that $\eta$ is odd in $w$ because $H$ is and $t\mapsto \Psi(\eta(t, w))$ is strictly decreasing by the properties of a pseudogradient.
\begin{lemma}
For each $\omega\in S\setminus K$, $\lim_{t\rightarrow T^{+}(\omega)}\eta(t, \omega)$ exists and is a critical point of $J$.
\end{lemma}
\begin{proof}
If $T^{+}(\omega)<\infty$ and let $0\leq s<t<T^{+}(\omega)$. Then
\begin{equation*}
\Vert \eta(t, \omega)-\eta(s, w)\Vert\leq\int_{s}^{t}\Vert H(\eta(\tau, w))\Vert d\tau\leq t-s.
\end{equation*}
Hence the limit exits and if it is not a critical point, then $\eta(\cdot, w)$ can be continued for $t>T^{+}(\omega)$.\\
Assume $T^{+}(\omega)=\infty$. It suffices to prove that for each $\epsilon>0$ there exists $t_{\epsilon}>0$ such that $\Vert \eta(t_{\epsilon}, \omega)-\eta(t, w)\Vert<\epsilon$ for any $t\geq t_{\epsilon}$. Argument by contradiction, we can find $\epsilon\in (0, \frac{\rho(d)}{2})$ and $\{t_{n}\}\subset \r^{+}$ with
$t_{n}\rightarrow +\infty$ and $\Vert \eta(t_{n}, \omega)-\eta(t_{n+1}, w)\Vert=\epsilon$ for all $n\geq 1$. Choose the smallest $t_{n}^{1}\in (t_{n}, t_{n+1})$
such that $\Vert \eta(t_{n}, \omega)-\eta(t_{n}^{1}, w)\Vert=\frac{\epsilon}{3}$ and let $\kappa_{n}:=\min\{z(\eta(s, \omega)): s\in[t_{n}, t_{n}^{1}]\}$. By the continuity of $\mu$, Proposition 3.6 and Proposition 7.1.1(viii) in \cite{rCh}, we have
\begin{align*}
\frac{\epsilon}{3}&=\Vert \eta(t_{n}, \omega)-\eta(t_{n}^{1}, w)\Vert\leq \int_{t_{n}}^{t_{n}^{1}}\Vert H(\eta(s, \omega))\Vert ds\leq t_{n}^{1}-t_{n}\\
&\leq \frac{2}{\kappa_{n}}\int_{t_{n}}^{t_{n}^{1}}\underset{\gamma\in \partial\Psi(\eta(s, u))}{\inf}\langle \gamma, H(\eta(s, \omega))\rangle ds=\frac{-2}{\kappa_{n}}\int_{t_{n}}^{t_{n}^{1}}\underset{\gamma\in \partial\Psi(\eta(s, u))}{\sup}\langle \gamma, -H(\eta(s, \omega))\rangle ds\\
&\leq \frac{-2}{\kappa_{n}}\int_{t_{n}}^{t_{n}^{1}}\frac{d}{ds}\Psi(\eta(s, \omega))ds=\frac{2}{\kappa_{n}}(\Psi(\eta(t_{n}, \omega)-\Psi(\eta(t_{n}^{1}, \omega)).
\end{align*}
Since $\Psi$ is bounded below, $\Psi(\eta(t_{n}, \omega)-\Psi(\eta(t_{n}^{1}, \omega)\rightarrow 0$, it follow that $\kappa_{n}\rightarrow 0$. Hence we can
find $s_{n}^{1}\in [t_{n}, t_{n}^{1}] $ such that $z(\eta(s_{n}^{1}, \omega))\rightarrow 0$ as $n\rightarrow\infty$. By the definition $\mu$ there exit
$\omega_{n}^{1}$ such that $\omega_{n}^{1}-\eta(s_{n}^{1}, \omega)\rightarrow 0$ and $\partial^{-}\Psi(\omega_{n}^{1})\rightarrow 0$. So $\lim\sup_{n\rightarrow\infty}\Vert \omega_{n}^{1}-\eta(t_{n}, \omega)\Vert\leq \frac{\epsilon}{3}$. Similarly, there exists a largest $t_{n}^{2}\in [t_{n}^{1}, t_{n+1}] $ such that $\Vert \eta(t_{n}^{2}, \omega)-\eta(t_{n+1}, w)\Vert=\frac{\epsilon}{3}$ and we can find
$\omega_{n}^{2}$ such that $\partial^{-}\Psi(\omega_{n}^{2})\rightarrow 0$ and $\lim\sup_{n\rightarrow\infty}\Vert \omega_{n}^{2}-\eta(t_{n+1}, \omega)\Vert\leq \frac{\epsilon}{3}$. So $\frac{\epsilon}{3}\leq \lim\sup_{n\rightarrow\infty}\Vert \omega_{n}^{1}-\omega_{n}^{2}\Vert\leq 2\epsilon<\rho(d)$, it contradicts with Lemma 3.11. The proof is completed.
\end{proof}
Let $P\subset S$, $\delta>0$ and define
\begin{equation*}
U_{\delta}(P):=\{w\in S: \mbox{dist}(w, P)< \delta\}.
\end{equation*}

\begin{lemma}
Let $d\geq c$. Then for every $\delta>0$ there exists $\epsilon=\epsilon(\delta)>0$ such that\\
  \item[(1)] $\Psi^{d+\epsilon}_{d-\epsilon}\cap K=K_{d}$\\
  \item[(2)] $\lim_{t\rightarrow T^{+}(w)}\Psi(\eta(t, w))<d-\epsilon$ for all $w\in \Psi^{d+\epsilon}\setminus U_{\delta}(K_{d})$.
\end{lemma}
\begin{proof}
Since we assume that $\mathcal{F}$ is a finite set, so (1) holds for $\epsilon>0$ small enough. Without loss of generality,
we assume that $U_{\delta}(K_{d})\subset J^{d+1}$ and $\delta<\rho(d+1)$. In order to find $\epsilon>0$ such that (2) holds, we let
\begin{equation*}
\tau: =\inf\Big\{\mu(\omega): \omega\in U_{\delta}(K_{d})\setminus U_{\frac{\delta}{2}}(K_{d})\Big\}
\end{equation*}
and claim that $\tau>0$. Argument by contradiction. Assume that there exits a sequence $\{v_{n}^{1}\}\subset U_{\delta}(K_{d})\setminus U_{\frac{\delta}{2}}(K_{d})$ such that $\mu(v_{n}^{1})\rightarrow 0$. According to the definition of $\mu$, there exists a PS-sequence $\{\omega_{n}^{1}\}$
of $\Psi$ such that $\Vert\omega_{n}^{1}- v_{n}^{1}\Vert\rightarrow 0$ as $n\rightarrow\infty$. Using this limit, the finiteness assumption of $\mathcal{F}$ and the $Z^{N}$-invariance of $\Psi$, we may assume that $\omega_{n}^{1}\in U_{\delta}(\omega_{0})\setminus U_{\frac{\delta}{2}}(\omega_{0})$ for some $\omega_{0}\in K_{d}$. Let $v_{n}^{2}\rightarrow \omega_{0}$. Since $\omega_{0}\in K_{d}$ and $\mu$ is continuous, $\mu(v_{n}^{2})\rightarrow 0$. As before, there exists a PS-sequence $\{\omega_{n}^{2}\}$
of $\Psi$ such that $\Vert\omega_{n}^{2}- v_{n}^{2}\Vert\rightarrow 0$ as $n\rightarrow\infty$, moreover, $\Vert\omega_{n}^{2}- \omega_{0}\Vert\rightarrow 0$ as $n\rightarrow\infty$.  So, we have
\begin{equation*}
\frac{\delta}{2}\leq\underset{n\rightarrow\infty}{\lim\sup}\Vert \omega_{n}^{1}-\omega_{n}^{2}\Vert\leq \delta<\rho(d+1),
\end{equation*}
this contradicts with Lemma 3.11.Hence $\tau$ is positive. Choose $\epsilon<\frac{\delta\tau}{4}$
such that (1) holds. By Lemma 3.12 and (1), the only way (2) can fail is that $\eta(t, w)\rightarrow \tilde{\omega}\in K_{d}$ as $t\rightarrow T^{+}(\omega)$
for some $w\in \Psi^{d+\epsilon}\setminus U_{\delta}(K_{d})$. In this case we let
\begin{equation*}
t_{1}:=\sup\{t\in[0, T^{+}(\omega)): \eta(t, w)\not\in U_{\delta}(\tilde{\omega})\}
\end{equation*}
and
\begin{equation*}
t_{2}:=\sup\{t\in(t_{1}, T^{+}(\omega)): \eta(t, w)\in U_{\frac{\delta}{2}}(\tilde{\omega})\}.
\end{equation*}
Then
\begin{equation*}
\frac{\delta}{2}=\Vert \eta(t_{1}, w)-\eta(t_{2}, w)\Vert \leq \int_{t_{1}}^{t_{2}}\Vert H(\eta(s, \omega))\Vert ds\leq t_{2}-t_{1},
\end{equation*}
and
\begin{align*}
\Psi(\eta(t_{1}, \omega)-\Psi(\eta(t_{2}, \omega)&=\int_{t_{1}}^{t_{2}}\frac{d}{ds}\Psi(\eta(s, \omega))ds\\
&\leq\int_{t_{1}}^{t_{2}}\underset{\gamma\in \partial\Psi(\eta(s, u))}{\sup}\langle \gamma, -H(\eta(s, \omega))\rangle ds\\
&\leq -\int_{t_{1}}^{t_{2}}\underset{\gamma\in \partial\Psi(\eta(s, u))}{\inf}\langle \gamma, H(\eta(s, \omega))\rangle ds\\
&\leq -\frac{1}{2}\mu(\eta(s, u))(t_{2}-t_{1})\\
&\leq -\frac{1}{2}\tau(t_{2}-t_{1})\leq -\frac{\delta\tau}{4}.
\end{align*}
Hence $\Psi(\eta(t_{2}, w))\leq d+\epsilon-\frac{\delta\tau}{4}<d$ and therefore $\eta(t, w)\not\rightarrow \tilde{\omega}$, it contradicts with our
assumption. This completes the proof.
\end{proof}

\begin{proof}[\textbf{Proof of Theorem 1.2}]
Let
\begin{equation*}
\Sigma :=\{A\subset S: A=\overline{A}, A=-A\}.
\end{equation*}
Recall that the definition of the Krasnoselskii genus $\gamma(A)$, for $A\subset \Sigma$ in \cite{rSt}. Define
\begin{equation*}
c_{k}:=\inf\{d\in \mathbb{R}: \gamma(\Psi^{d})\geq k\}, \quad k\geq 1.
\end{equation*}
Thus $c_{k}$ are those numbers at which the set $\Psi^{d}$ change genus and it is easy to see that $c_{k}\leq c_{k+1}$. We claim:
\begin{equation*}
K_{c_{k}}\neq \emptyset \quad \text{and}\quad c_{k}<c_{k+1}\quad \text{for all}\, k\in \mathbb{N}.
\end{equation*}
To prove this, let $k\geq 1$ and set $d:=c_{k}$. By Lemma 3.10, $K_{d}$ is either empty or a discrete set, hence $\gamma(K_{d})=0$ or 1. By the
continuity property of the genus, there exists $\delta>0$ such that $\gamma(\overline{U})=\gamma(K_{d})$, where $U:=U_{\delta}(K_{d})$
and $\delta< \frac{\kappa}{2}$. For such $\delta$, choose $\epsilon>0$ so that the conclusions of Lemma 3.13 hold. Then for each
$w\in \Psi^{d+\epsilon}\setminus U$ there exists $t\in [0, T^{+}(w))$ such that $\Psi(\eta(t, w))\leq d-\epsilon$. Let $e=e(w)$ be the infimum
of the time for which $\Psi(\eta(t, w))\leq d-\epsilon$. Since $d-\epsilon$ is not a critical value of $\Psi$, it is easy to see by the Implicit
Function Theorem that $e$ is a continuous mapping and since $\Psi$ is even, $e(-w)=e(w)$. Define a mapping $h:\Psi^{d+\epsilon}\setminus U\rightarrow \Psi^{d+\epsilon}$
by setting $h(w):=\eta(e(w), w)$. Then $h$ is odd and continuous, so it follows from the properties of the genus and the definition of $c_{k}$ that
\begin{equation*}
\gamma(\Psi^{d+\epsilon})\leq \gamma(\overline{U})+\gamma(\Psi^{d-\epsilon})\leq \gamma(\overline{U})+k-1=\gamma(K_{d})+k-1.
\end{equation*}
If $\gamma(K_{d})=0$, then $\gamma(\Psi^{d+\epsilon})\leq k-1$, it contradicts the definition of $c_{k}$. So $\gamma(K_{d})=1$ and $K_{d}\neq \emptyset$.
If $c_{k}= c_{k+1}=d$, then $\gamma(K_{d})>1$. Since this is impossible, we must have $c_{k}\leq c_{k+1}$ and $K_{c_{k}}\neq \emptyset$ for all $k\geq 1$. Hence, the proof is finished.
\end{proof}

\section{Proof of Theorem 1.3} \label{pf2}
In this section, we assume that $V(x)$ is coercive, that is, $V(x)\rightarrow +\infty$ as $\vert x\vert \rightarrow\infty$. To prove Theorem 1.3, we  need to adapt the proof of Theorem 1.1. The main difference between them is how to show that the solution is nontrivial.
From section 3, we know that Lemma 3.1 is very important. By a simple observation, in addition to Lemma 3.1(vi), the proof of other results in Lemma 3.1 are the same as the coercive potential case.
\begin{lemma}
 Assume that $(F_{1})-(F_{4})$ hold and $V(x)\rightarrow +\infty$ as $\vert x\vert \rightarrow\infty$, then $J$ is coercive on $E$ ,i.e., $J(u)\rightarrow \infty$, as $\Vert u\Vert\rightarrow\infty$.
\end{lemma}
\begin{proof}
Arguing by contraction, suppose there exists a sequence $\{\omega_{n}\}\subset\mathcal{N}$ such that $\Vert \omega_{n}\Vert\rightarrow \infty$ and
$J(\omega_{n})\leq d$ for some $d>0$. Let $v_{n}=\frac{\omega_{n}}{\Vert \omega_{n}\Vert }$. Then $v_{n}\rightharpoonup v$ in $E$,  $v_{n}\rightarrow v$ in $L^{p}(\rn)$ and $v_{n}(x)\rightarrow v$ a.e. $x\in\rn$,
up to a subsequence. If $v=0$, then by (3.1),
$\irn F(x, sv_{n})dx\rightarrow 0$ for all $s\in\r$ and therefore
\begin{align*}
d\geq J(\omega_{n})\geq J(sv_{n})=\frac{s^{2}}{2}-\irn F(x, sv_{n})dx\rightarrow \frac{s^{2}}{2},
\end{align*}
this yields a contraction if $s>\sqrt{2d}$. So $v_{n}(x)\rightarrow v\neq 0$ a.e. $x\in\rn$. Moreover, by $(F_{3})$ one has
\begin{align*}
0\leq \frac{J(\omega_{n})}{\Vert \omega_{n}\Vert^{2} }=\frac{1}{2}-\irn \frac{ F(x, \omega_{n})}{\Vert \omega_{n}\Vert^{2}}dx=\frac{1}{2}-\irn \frac{ F(x,  \omega_{n}) }{\omega_{n}^{2}}v_{n}^{2}dx\rightarrow -\infty.
\end{align*}
This yields a contradiction.
\end{proof}

\begin{proof}[\textbf{Proof of Theorem 1.3}]
Similar with the proof of Theorem 1.1, there exists a bounded PS-sequence $\{\omega_{n}\}$ of $J$. Up to a subsequence,  $\omega_{n}\rightharpoonup u$, $\omega_{n}\rightarrow u$ in $L^{p}(\rn)$
 and $\omega_{n}\rightarrow u$ in a.e. $x\in \rn$.
 If $u=0$, then $\irn F(x, \omega_{n})dx\rightarrow 0$ and $\irn f(x, \omega_{n})\omega_{n}dx\rightarrow 0$ as $n\rightarrow\infty$, this implies that $\Vert \omega_{n}\Vert\rightarrow 0$ as $n\rightarrow\infty$ and this contradicts with $\Vert \omega\Vert$ bounded away from 0, for any $\omega\in \mathcal{N}$, so $u\neq 0$ is a nontrivial solution. Moreover, by Fatou's lemma and (3.2), it is easy to know that $u$ is a ground state solution.
\end{proof}

\section{Proof of Theorem 1.4} \label{pf2}
We assume that  $\inf_{x\in \rn}V(x)\leq V(x)<\lim_{\vert x\vert\rightarrow \infty}V(x)=\sup_{x\in \rn}V(x)<+\infty$ and $f(x, u)=f(u)$ hold in this section. As in Section 4, except for the proof of Lemma 3.1(vi), others are the same. Now we give its proof in the bounded potential well case.
\begin{lemma}
Assume that $f(x, u)=f(u)$,  $(F_{1})-(F_{4})$ and  $\inf_{x\in \rn}V(x)\leq V(x)<\lim_{\vert x\vert\rightarrow \infty}V(x)=\sup_{x\in \rn}V(x)<+\infty$  hold, then   $J$ is coercive on $E$ ,i.e., $J(u)\rightarrow \infty$, as $\Vert u\Vert\rightarrow\infty$.
\end{lemma}
\begin{proof}
 Arguing by contraction, suppose there exists a sequence $\{\omega_{n}\}\subset\mathcal{N}$ such that $\Vert \omega_{n}\Vert\rightarrow \infty$ and
$J(\omega_{n})\leq d$ for some $d>0$. Let $v_{n}=\frac{\omega_{n}}{\Vert \omega_{n}\Vert }$. Then $v_{n}\rightharpoonup v$ in $E$ and $v_{n}(x)\rightarrow v$ a.e. $x\in\rn$,
up to a subsequence. For some $R>0$, choose $y_{n}\in\rn$ satisfy \\
\begin{align*}
\int_{B_{R}(y_{n})}v_{n}^{2}dx=\underset{y\in\rn}{\sup}\int_{B_{R}(y)}v_{n}^{2}dx.
\end{align*}
Similar with the proof of Lemma 3.1(v), we may prove that there exists $\delta>0$ such that
\begin{align}
\int_{B_{R}(y_{n})}v_{n}^{2}dx\geq \delta.
\end{align}
Set $\widetilde{v}_{n}(\cdot)=v_{n}(\cdot-y_{n})$, then we have $\widetilde{v}_{n}\rightharpoonup \widetilde{v}$ in $E$, $\widetilde{v}_{n}\rightarrow \widetilde{v}$ in $L_{loc}^{2}(\rn)$, and $\widetilde{v}_{n}\rightarrow \widetilde{v}$ in a.e. $x\in \rn$. In virtue of (5.1), we have $\widetilde{v}\neq 0$.
By $(F_{3})$ we obtain that
\begin{align*}
0\leq \frac{J(u_{n})}{\Vert u_{n}\Vert^{2} }=\frac{1}{2}-\irn \frac{ F(u_{n})}{\Vert u_{n}\Vert^{2}}dx=\frac{1}{2}-\irn \frac{ F(u_{n}(x-y_{n})) }{\vert u_{n}(x-y_{n})\vert^{2}}\widetilde{v}_{n}^{2}dx\rightarrow -\infty.
\end{align*}
This is a contradiction. we complete the proof.
\end{proof}
 We shall need a limiting problem
\begin{equation}  \label{1}
(-\Delta)^{\alpha} u+V_{\infty}u =f(u), \quad  x\in \rn.
\end{equation}
The energy functional corresponding to it is
\begin{equation*} \label{fcl}
J_{\infty}(u):= \frac{1}{2}\irn(\vert  (-\Delta)^{\frac{\alpha}{2}} u\vert^{2}+V_{\infty}u^{2})dx-\irn F(u)\,dx.
\end{equation*}
Let
\begin{equation*}
\mathcal{N}_{\infty}=\{u\in E\setminus\{0\}: \irn\vert  (-\Delta)^{\frac{\alpha}{2}} u\vert^{2}+V_{\infty}u^{2}dx=\irn f(u)u\,dx\}
\end{equation*}
be the Nehari manifold for $J_{\infty}$. Since $V_{\infty}$ be constant and $f$ independs on $x$, there exists a solution $u_{\infty}\neq 0$ for
minimizes $J_{\infty}$ on $\mathcal{N}_{\infty}$ by Theorem 1.1.

\begin{lemma} \label{compar}
(i) If $V< V_\infty$, then $0<c < c_\infty$, where $c_{\infty}:=\underset{u\in \mathcal{N}_{\infty}}{\inf}J_{\infty}(u)$. \\
(ii) For $\{u_{n}\}\subset \mathcal{N}$, if $J(u_n)\to d\in (0, c_\infty)$ and $J'(u_n)\to 0$, then $u_n\rh u\ne 0$ after passing to a sub\-sequence, $u$ is a critical point of $J$ and $J(u)\le d$.
\end{lemma}

\begin{proof}
(i) Let $s_0>0$ be such that $s_0u_\infty\in \cn$. Since $V(x) < V_\infty$ in $x\in \rn$, we have
\[
c \le J(s_0u_\infty) < J_\infty(s_0u_\infty) \le J_\infty(u_\infty) = c_\infty.
\]
(ii) Because $J$ is coercive on $\mathcal{N}$,  $\{u_{n}\}$ is bounded. Up to a subsequence, $u_n\rh u$ in $E$, $u_n(x)\to u(x)$ a.e. $x\in \rn$. By Fatou's lemma,
\begin{align*}
d & = J(u_n) -\frac12\langle J'(u_n),u_n\rangle +o(1) = \irn \Big(\frac12 f(u_n)u_n-F(u_n)\Big)\,dx + o(1) \\
& \ge \irn \Big(\frac12 f(u)u-F(u)\Big)\,dx+ o(1) = J(u) - \frac12\langle J'(u),u\rangle +o(1) = J(u)+o(1).
\end{align*}
So $J(u)\le d$ and it remains to show that $u\ne 0$. Arguing indirectly, suppose $u=0$. Since $u_n\to 0$ in $L^2_{loc}(\rn)$ and $V(x)\to V_\infty$ as $|x|\to\infty$,
\[
J(u_n)-J_\infty(u_n) = \frac12\irn(V(x)-V_\infty)u_n^2\,dx \to 0
\]
and therefore $J_\infty(u_n)\to d$.
Using the H\"older and the Sobolev inequalities and taking $v$ with $\|v\|=1$, we obtain
\begin{align*}
\left|\langle J'(u_n)-J'_\infty(u_n), v\rangle \right| & \le \irn (V_\infty-V(x)|u_n|\,|v|\,dx \\
& \le C\left(\irn(V_\infty-V(x))u_n^2\,dx\right)^{1/2}.
\end{align*}
As the right-hand side tends to 0 uniformly in $\|v\|=1$, $J'(u_n)-J'_\infty(u_n)\to 0$ and hence $J'_\infty(u_n)\to 0$. So
\begin{equation*}
0 = \langle J'(u_n),u_n\rangle \geq \frac{1}{2}\|u_n\|^2-C_1\int_{\{|u_n|\geq 1\}}  |u_n|^{p}\,dx
 \end{equation*}
and if $\|u_n\|_p\to 0$, then $u_n\to 0$ in $E$ which is impossible because $J(u_n)\to d > 0$.
Hence by Proposition 2.2, for some $R>0$,  there are $(y_n)\subset\rn$ and $\delta>0$ such that
\[
\int_{B_R(y_n)}u_n^2\,dx \ge \delta.
\]
Let $v_n(x) := u_n(x+y_n)$. Since $J_\infty$ is invariant with respect to translations by elements of $\rn$, $J_\infty(v_n)\to d$ and $J'_\infty(v_n) \to 0$. Moreover,
\[
\int_{B_1(0)}v_n^2\,dx = \int_{B_1(y_n)}u_n^2\,dx \ge \delta
\]
and therefore $v_n\rh v\ne 0$ after passing to a subsequence. It follows that $v$ is a nontrivial critical point of $J_\infty$ and $J_\infty(v)\le d < J_\infty(v_\infty)$ which is the desired contradiction.
\end{proof}

\begin{proof}[\textbf{Proof of Theorem 1.4}]
Similar with the proof of Theorem 1.1, there exists a bounded sequence $\{\omega_{n}\}$ such that $J(\omega_{n})\rightarrow c$  and
$J'(\omega_{n})\rightarrow 0$. Using Lemma 5.2 we obtain a critical point $u\neq 0$ of $J$ such that $J(u)\leq c$. So $J(u)=c$ and $u$
is a ground state solution of problem (1.1). The proof is completed.
\end{proof}

\section{Proof of Theorem 1.5} \label{pf2}
Now we seek the ground state solutions of problem (1.1) when $V$ and $f$ are asymptotically periodic in $x$. Firstly, by a simple observation,
Lemma 3.1 holds under assumptions $(V_{1})$ and $(F_{1})-(F_{4})$. Moreover, to prove Theorem 1.5, we need some lemmas.
\begin{lemma}
Assume $(V_{1})$ and $(F_{1})-(F_{5})$ hold. Then $J(u)\leq J_{p}(u)$, for all $u\in E$.
\end{lemma}
It follows an easy inspection.
\begin{lemma}
Assume $(V_{1})$ and $(F_{5})(ii)$ hold. Assume that $\{u_{n}\}\subset E$ satisfies $u_{n}\rightharpoonup 0$ and $\varphi_{n}\in E$ is bounded. Then
\begin{align}
\irn (V(x)-V_{p}(x))u_{n}\varphi_{n}dx\rightarrow 0,
\end{align}
\begin{align}
\irn (f(x, u_{n})-f_{p}(x, u_{n}))\varphi_{n}dx\rightarrow 0,
\end{align}
\begin{align}
\irn (F(x, u_{n})-F_{p}(x, u_{n}))dx\rightarrow 0.
\end{align}
\end{lemma}
For the proof of this lemma one may refer to \cite{rZXZ}, so we omit it.\\

\begin{proof}[\textbf{Proof of Theorem 1.5}]
As the same in the proof of Theorem 1.1, there exists a bounded PS-sequence $\{\omega_{n}\}\subset m(u_{n})$ satisfies $J(\omega_{n})\rightarrow c$
and $J'(\omega_{n})\rightarrow 0$. Up to a subsequence, $\omega_{n}\rightharpoonup u$ in $E$, $\omega_{n}\rightarrow u$ in $L_{loc}^{2}(\rn)$, and $\omega_{n}\rightarrow u$ a.e. on $x\in\rn$. If $u\neq 0$,  $u$ is a ground state solution of problem (1.1) and the proof is completed. Now we show that $u\neq 0$. Arguing by contradiction, if $\omega_{n}\rightarrow 0$ in $L^{p}(\rn)$, then $\irn F(x, \omega_{n})dx\rightarrow$ and $\irn f(x, \omega_{n})\omega_{n}dx\rightarrow 0$ as $n\rightarrow\infty$, this implies that $\Vert \omega_{n}\Vert\rightarrow 0$ as $n\rightarrow\infty$ and this contradicts with Lemma. If $\omega_{n}\not\rightarrow 0$ in $L^{p}(\rn)$, by Proposition 2.2, for some $R>0$ and $\delta>0$, there exist $y_{n}$ such that
 \begin{align}
\int_{B_{r}(y_{n})}\omega_{n}^{2}dx\geq \delta.
\end{align}
 Without loss of generality, we assume that $y_{n}\in Z^{N}$. Setting $\bar{\omega}_{n}(x)=\omega_{n}(x-y_{n})$, up to a subsequence,
 we have $\bar{\omega}_{n}\rightharpoonup \bar{u}$ in $E$, $\bar{\omega}_{n}\rightarrow \bar{u}$ in $L_{loc}^{2}(\rn)$, and $\bar{\omega}_{n}\rightarrow \bar{u}$ a.e. on $x\in\rn$. By Fatou's lemma and (6.4), $\bar{u}\neq 0$.\\
 For any $\varphi \in E$, set $\varphi_{n}(\cdot)=\varphi(\cdot-y_{n})$, by (6.1) and (6.2) in Lemma 6.2 , we may obtain
 \begin{align*}
\langle J'(\omega_{n}), \varphi_{n}\rangle-\langle J_{p}'(\omega_{n}), \varphi_{n}\rangle\rightarrow 0.
\end{align*}
 Since $J'(\omega_{n})\rightarrow 0$ and $\Vert \varphi_{n}\Vert=\Vert \varphi\Vert$, $\langle J'(\omega_{n}), \varphi_{n}\rangle\rightarrow 0$, so we have
 $\langle J_{p}'(\omega_{n}), \varphi_{n}\rangle\rightarrow 0$.  Because $V_{p}$ and $f_{p}$ are 1-periodic in and $y_{n}\in Z^{N}$, one has
 $\langle J_{p}'(\omega_{n}), \varphi_{n}\rangle=\langle J_{p}'(\bar{\omega_{n}}), \varphi\rangle$. Since $\varphi$ is arbitrary, $J_{p}'(\bar{\omega_{n}})\rightarrow 0$ in $E$ as $n\rightarrow \infty$. Since $J'_{p}$ is weakly sequently continuous by, we have $J'_{p}(\bar{u})=0$.\\

 Now we show that $J_{p}(\bar{u})\leq c$. Replacing $\varphi_{n}$ by $\omega_{n}$ in Lemma 6.2, we have
 \begin{align}
\irn (f(x, \omega_{n})-f_{p}(x, \omega_{n}))\omega_{n}dx\rightarrow 0.
\end{align}
Combine with  (6.5) and (6.3), we have
 \begin{align}
\irn \Big(\frac{1}{2}f_{p}(x, \omega_{n})\omega_{n}-F_{p}(x, \omega_{n})\Big)dx=\irn \Big(\frac{1}{2}f(x, \omega_{n})\omega_{n}-F(x, \omega_{n})\Big)dx+o_{n}(1).
\end{align}
Since $\frac{1}{2}f_{p}(x, \omega_{n})\omega_{n}-F_{p}(x, \omega_{n})$ is 1-periodic in $x_{1},\ldots, x_{N}$, so we have
 \begin{align}
\irn \Big(\frac{1}{2}f_{p}(x, \omega_{n})\omega_{n}-F_{p}(x, \omega_{n})\Big)dx=\irn \Big(\frac{1}{2}f_{p}(x, \bar{\omega}_{n})\bar{\omega}_{n}-F_{p}(x, \bar{\omega}_{n})\Big)dx.
\end{align}
By (6.6) and (6.7), one has
 \begin{align*}
\irn \Big(\frac{1}{2}f_{p}(x, \bar{\omega}_{n})\bar{\omega}_{n}-F_{p}(x, \bar{\omega}_{n})\Big)dx=\irn \Big(\frac{1}{2}f(x, \omega_{n})\omega_{n}-F(x, \omega_{n})\Big)dx+o_{n}(1).
\end{align*}
By and Fatou's lemma, it follows that
 \begin{align*}
\underset{n\rightarrow \infty}{\lim}\irn \Big(\frac{1}{2}f_{p}(x, \bar{\omega}_{n})\bar{\omega}_{n}-F_{p}(x, \bar{\omega}_{n})\Big)dx\geq \irn \Big(\frac{1}{2}f_{p}(x, \bar{u})\bar{u}-F_{p}(x, \bar{u})\Big)dx,
\end{align*}
so
 \begin{align*}
\underset{n\rightarrow \infty}{\lim}\irn \Big(\frac{1}{2}f(x, \omega_{n})\omega_{n}-F(x, \omega_{n})\Big)dx\geq \irn \Big(\frac{1}{2}f_{p}(x, \bar{u})\bar{u}-F_{p}(x, \bar{u})\Big)dx.
\end{align*}
Moreover, we have
 \begin{align*}
c:=\underset{n\rightarrow \infty}{\lim}J(\omega_{n})&=\underset{n\rightarrow \infty}{\lim}\Big(J(\omega_{n})-\frac{1}{2}J'(\omega_{n})\omega_{n}\Big)\\
&=\underset{n\rightarrow \infty}{\lim}\irn \Big(\frac{1}{2}f(x, \omega_{n})\omega_{n}-F(x, \omega_{n})\Big)dx\\
&\geq \Big(\frac{1}{2}f_{p}(x, \bar{u})\bar{u}-F_{p}(x, \bar{u})\Big)dx.\\
&=J_{p}(\bar{u})-\frac{1}{2}J'_{p}(\bar{u})\bar{u}=J_{p}(\bar{u}).
\end{align*}
Since $J'_{p}(\bar{u})=0$ and $\bar{u}\neq 0$, by Lemma 3.1(iii), $\underset{t\geq 0}{\max}J_{p}(t\bar{u})=J_{p}(\bar{u})$ and there exists $t_{\bar{u}}>0$ such that $t_{\bar{u}}\bar{u}\in \mathcal{N}$. Then
\begin{align*}
J(t_{\bar{u}}\bar{u})\leq J_{p}(t_{\bar{u}}\bar{u})\leq \underset{t\geq 0}{\max}J_{p}(t\bar{u})=J_{p}(\bar{u})
\end{align*}
In virtue of $J_{p}(\bar{u})\leq c$, $J(t_{\bar{u}}\bar{u})\leq c$. So $J(t_{\bar{u}}\bar{u})=c$. The proof is completed.
\end{proof}

\end{document}